\documentclass[12pt]{amsart}
\usepackage{geometry}
\usepackage{amsmath}
\numberwithin{equation}{section}
\numberwithin{table}{section}

\usepackage{amsfonts}
\usepackage{amssymb}
\usepackage{amsthm}
\usepackage{amstext}
\usepackage{array}
\usepackage{url}
\usepackage[usenames,dvipsnames]{color}
\usepackage[normalem]{ulem}
\usepackage{enumitem}
\usepackage{hyperref}
\usepackage[all,arc,curve,frame,color]{xy}

\usepackage{cite} 
\definecolor{dkgreen}{rgb}{0,0.6,0}
\definecolor{gray}{rgb}{0.5,0.5,0.5}
\definecolor{mauve}{rgb}{0.58,0,0.82}

\usepackage{placeins}

\theoremstyle{plain}
\newtheorem{thm}{Theorem} 

\theoremstyle{plain}
\newtheorem{lem}{Lemma}[section]
\newtheorem{cor}[lem]{Corollary}
\newtheorem{prop}[lem]{Proposition}

\newtheorem*{remarks*}{Remarks}
\newtheorem{rem}[lem]{Remark}
\newtheorem*{rem*}{Remark}

\theoremstyle{definition}
\newtheorem{defn}{Definition}

\newtheorem{conj}[defn]{Conjecture}

\newtheorem{qstn}[defn]{Question}

\newcommand{\magma}{{\sc magma }}

\newcommand{\QQ}{\mathbb{Q}}
\newcommand{\RR}{\mathbb{Q}}
\newcommand{\PP}{\mathbb{P}}
\renewcommand{\AA}{\mathbb{A}}
\newcommand{\PGL}{\textrm{PGL}}
\newcommand{\PrePer}{\textrm{PrePer}}
\newcommand{\Per}{\textrm{Per}}
\newcommand{\Ind}{\textrm{Ind}}

\title{Quadratic maps with a periodic critical point of period~2}

\author{Jung Kyu Canci}
\email{jungkyu.canci@unibas.ch}
\author{Solomon Vishkautsan} 
\thanks{The second author is supported by the ERC-Grant ``Diophantine Problems," No.\ 267273.}
\email{wishcow@gmail.com \textrm{or} solomon.vishkautsan@sns.it}

\makeatletter
\let\thetitle\@title
\let\theauthor\@author
\makeatother

\begin{document}
\maketitle

\begin{abstract}
We provide a complete classification of possible graphs of rational preperiodic points of endomorphisms of the projective line of degree 2 defined over the rationals with a rational periodic critical point of period 2, under the assumption that these maps have no periodic points of period at least 7. We explain how this extends results of Poonen on quadratic polynomials. We show that there are exactly 13 possible graphs, and that such maps have at most 9 rational preperiodic points. We provide data related to the analogous classification of graphs of endomorphisms of degree 2 with a rational periodic critical point of period 3 or 4.
\end{abstract}

\section{Introduction}
Let $\phi:\mathbb{P}^1\rightarrow\mathbb{P}^1$ be an endomorphism defined over a field $K$. A point $P\in\mathbb{P}^1$ is \emph{periodic} for $\phi$ if $\phi^n(P)=P$ for some $n\geq{1}$ and the minimal such $n$ is called the \emph{period} of $P$. A point $P\in\mathbb{P}^1$ is called \emph{preperiodic} if some iterate of $P$ is periodic, i.e.\ there exists an $m\geq{0}$ such that $\phi^m(P)$ is periodic. We denote by $\PrePer(\phi, K)$ the set of preperiodic points for $\phi$ in $\mathbb{P}^1(K)$ (similarly $\Per(\phi,K)$ is the set of $K$-periodic points and $\Per_n(\phi,K)$ is the set of $K$-periodic points of period $n$). By a classical theorem of Northcott~\cite{article:northcott1950}, the set $\PrePer(\phi, K)$ is finite for a number field $K$; it can therefore be given a finite directed graph structure, the \emph{preperiodicity graph} of $\phi$, denoted by $G_\phi$, by drawing an arrow from $P$ to $\phi(P)$ for each $P\in \PrePer(\phi, K)$. It is a natural question to ask which types of graphs (up to graph isomorphism) can be obtained from such maps. It is not known whether the list of possible graphs is finite; this is equivalent to the ($1$-dimensional) Uniform Boundedness Conjecture of Morton and Silverman~\cite{article:morton-silverman1994} that says $\#\PrePer(\phi, K)$ is bounded by a bound depending only on the degrees of $K/\mathbb{Q}$ and $\phi$.

The simplest case of the graph classification question is for quadratic polynomials defined over $\mathbb{Q}$. Flynn, Poonen and Schaefer~\cite{article:flynn-et-al1997} conjectured that for any integer $N>3$ there is no quadratic polynomial with coefficients in $\mathbb{Q}$ with a $\mathbb{Q}$-periodic point of period $N$. Assuming this conjecture is true, Poonen~\cite{article:poonen1998} provided a complete classification of $12$ possible preperiodicity graphs for quadratic polynomials defined over $\mathbb{Q}$. Another consequence of Poonen's classification is that the number of $\mathbb{Q}$-preperiodic points of a quadratic polynomial is at most $9$. The Flynn, Poonen and Schaefer conjecture was proved for the special cases of a periodic point of period $N=4$ (Morton~\cite{article:morton1998}), $N=5$ (Flynn, Poonen and Schaefer~\cite{article:flynn-et-al1997}) and $N=6$ (Stoll~\cite{article:stoll2008}, depending on the Birch and Swinnerton-Dyer conjecture); experimental results by Hutz and Ingram~\cite{article:hutz-ingram2013} and Benedetto et al.~\cite{article:benedetto-et-al2014} provide further evidence for it.

In this paper we study a natural generalization of quadratic polynomials, which are endomorphisms of $\mathbb{P}^1$ of degree $2$ (henceforth called a \emph{quadratic map}) having a periodic \emph{critical} point (i.e., a periodic point that is a ramification point of the endomorphism). This is indeed a generalization of a quadratic polynomial, since if the periodic critical point is fixed (i.e., of period $1$) then the map is exactly a quadratic polynomial (up to a conjugation by a projective automorphism of $\mathbb{P}^1$). 

For a quadratic map defined over $\mathbb{Q}$ with a $\mathbb{Q}$-rational periodic critical point of period $2$, we have a complete classification of possible preperiodicity graphs, assuming a conjecture similar to that of Flynn, Poonen and Schaefer~\cite{article:flynn-et-al1997}.

\begin{conj} \label{conj:main2}
Let $\phi$ be a quadratic map defined over $\mathbb{Q}$ with a $\mathbb{Q}$-rational periodic critical point of period $2$, then $\phi$ has no $\mathbb{Q}$-periodic point of period greater or equal to $3$.
\end{conj}

\begin{thm} \label{thm:main-thm1}
Assuming Conjecture \ref{conj:main2}, there are exactly $13$ possible preperiodicity graphs for quadratic maps defined over $\mathbb{Q}$ with a $\mathbb{Q}$-rational periodic critical point of period $2$; these are listed in Tables \ref{table:r2-pcf} and \ref{table:r2}. Moreover, the number of preperiodic points of such maps is at most $9$ (as in the quadratic polynomial case).
\end{thm}

We provide evidence for Conjecture~\ref{conj:main2} in the next theorem.

\begin{thm} \label{thm:main-thm2}
Let $\phi$ be a quadratic map defined over $\mathbb{Q}$ with a $\mathbb{Q}$-rational periodic critical point of period $2$, then $\phi$ has no $\mathbb{Q}$-periodic point of period $3,4,5$ or $6$. 
\end{thm}

Following Faber~\cite{article:faber2015}, we say that a finite directed graph $G$ is \emph{admissible} (by quadratic maps defined over $\mathbb{Q}$) if there exists a quadratic map $\phi$ defined over $\mathbb{Q}$ such that $\PrePer(\phi, \mathbb{Q}) \supseteq G$ (where by $\supseteq$ we mean contains an isomorphic subgraph; we will also say $\phi$ \emph{admits} $G$). Similarly, a graph $G$ is \emph{inadmissible} if no such quadratic map exists. We say that a graph $G$ is \emph{realizable} (by quadratic maps defined over $\mathbb{Q}$) if $\PrePer(\phi, \mathbb{Q}) = G$ for some quadratic map $\phi$ defined over $\mathbb{Q}$ (and that $\phi$  \emph{realizes} $G$). The endomorphisms in Tables~\ref{table:r2}, \ref{table:r3} and \ref{table:r4} were verified to realize the corresponding graphs by using an algorithm of Hutz~\cite{article:hutz2015} that determines all $\mathbb{Q}$-preperiodic points of a given endomorphism of $\mathbb{P}^n$.

An endomorphism of $\mathbb{P}^1$ is called \emph{post-critically finite} (\emph{PCF}) if all of its critical points are preperiodic. Lukas, Manes and Yap~\cite{article:lukas-manes-yap2014} provided a complete classification (over $\overline{\mathbb{Q}}$) of the possible graphs realized by PCF-quadratic maps defined over $\mathbb{Q}$. We list all the  preperiodicity graphs of PCF-quadratic maps defined over $\mathbb{Q}$ with a $\QQ$-rational periodic critical point of period $2$ in Table~\ref{table:r2-pcf}. The examples in the table are taken directly from~\cite{article:lukas-manes-yap2014}. Therefore we are left with the task of classifying the graphs realizable by \emph{non}-PCF endomorphisms with a $\mathbb{Q}$-rational periodic critical point of period $2$.

We say that two rational functions $\phi,\psi:\mathbb{P}^1\rightarrow\mathbb{P}^1$ are \emph{linearly conjugate} if there exists an $f\in \PGL_2$ acting as a projective automorphism of $\mathbb{P}^1$ such that $\psi = \phi^f = f^{-1}\phi{f}$. Two conjugate rational functions have the same dynamical behavior, since if $P$ is periodic for $\phi$ of period $n$, then $f^{-1}(P)$ is $\psi$-periodic of period $n$ (similarly for preperiodic points). When $\psi, \phi$ and $f$ are all defined over the same number field $K$, then it is clear from this that $G_\phi=G_\psi$. Therefore when classifying realizable graphs we are actually interested in the conjugacy classes of quadratic maps rather than in the individual maps.

We prove Theorem~\ref{thm:main-thm1} by showing that the graphs in Table~\ref{table:n2e} are inadmissible. This suffices to prove that no other graph other than those in Tables~\ref{table:r2-pcf} and~\ref{table:r2} is realizable by a quadratic map defined over $\mathbb{Q}$ with a $\mathbb{Q}$-rational periodic critical point of period $2$. The idea is to construct an affine curve $C(G)$ for each graph $G$ in Table~\ref{table:n2e}, whose points parametrize $\PGL_2(\QQ)$ conjugacty classes of quadratic maps with a $\mathbb{Q}$-rational periodic critical point of period $2$ that admit $G$ and show that these curves have no $\mathbb{Q}$-rational points. In order to obtain the graphs in the tables, we used a recursive algorithm described in Lemma~\ref{lem:recursive} to generate a list of potentially realizable graphs (roughly eighty) having up to $14$ vertices (we first assumed Benedetto et al.'s conjecture~\cite{article:benedetto-et-al2014} that a quadratic map defined over $\mathbb{Q}$ has at most $14$ $\mathbb{Q}$-preperiodic points, however we do not rely on this conjecture in the proofs). This produces a Hasse diagram of graphs (the graphs are partially ordered by the isomorphic subgraph relation, see Proposition~\ref{prop:hasse} below); we go up the diagram constructing the curves $C(G)$ until we reach graphs which are inadmissible (one can visualize this as trimming an infinite tree until we are left with a finite tree).

\begin{qstn} \label{qstn:infinitely-realizable}
Except for graph R2P4 in Table~\ref{table:r2} which has a unique class realizing it (see Proposition~\ref{prop:R2P4} below), are there infinitely many conjugacy classes realizing each of the graphs in Table~\ref{table:r2}?
\end{qstn}

Assuming Conjecture~\ref{conj:main2}, it is clear that the graphs R2P5, R2P6, R2P7 in Table~\ref{table:r2} have infinitely many classes realizing them, since the relevant curves $C(G)$ have infinitely many rational points (since they have genus $0$ and a rational point, see Corollary~\ref{cor:genus0} below) and these graphs are maximal with respect to the isomorphic subgraph relation among the realizable graphs (see the Hasse diagram in Proposition~\ref{prop:hasse} below). However, one can hope to prove that all the graphs except for R2P4 are realizable by infinitely many classes \emph{without} relying on this conjecture by adapting the technique of Faber~\cite{article:faber2015} who proved a similar result for quadratic polynomials. 

The cases of quadratic maps defined over $\mathbb{Q}$ with a $\mathbb{Q}$-rational periodic critical point of period $3$ or $4$ are more complicated, as some of the curves parametrizing quadratic maps with a given preperiodicity graph have a large genus. The $11$ graphs which we know to be realizable by such quadratic maps (R3P0 is the only one that is PCF) are listed in Tables~\ref{table:r3} and \ref{table:r4}. 

\begin{qstn} \label{qstn:3-4complete}
Are the Tables~\ref{table:r3} and \ref{table:r4} complete? That is, are the graphs listed in these tables the only possible preperiodicity graphs of quadratic maps defined over $\mathbb{Q}$ with a $\mathbb{Q}$-rational periodic critical point of period $3$ or $4$?
\end{qstn}

Let $P$ be a periodic point of exact period $n$ for an endomorphism $\phi:\mathbb{P}^1\rightarrow\mathbb{P}^1$. The \emph{multiplier} of $\phi$ at $P$ is $\lambda_P(\phi) = (\phi^n)'(P)$ (we are assuming none of the iterates of $P$ are $\infty$; we can always make sure this is true by a linear conjugation). We say that a periodic point $P$ belongs to \emph{critical cycle} if its multiplier is $0$; this is equivalent to some iteration of $P$ being a critical point. In the preperiodicity graph of a non-PCF quadratic map, the sequence of iterations of $P$ is a simple cycle, all of whose vertices but one have \emph{indegree} $2$ (indegree is the number of arrows leading into a vertex, denoted by $\deg^{-}$), and exactly one vertex with indegree $1$. The vertex in the critical cycle with indegree $1$ is the image of the critical point in the cycle, and is called the \emph{critical image} of the critical point. Using these definitions, we now expand the scope of Conjecture~\ref{conj:main2}.

\begin{qstn} \label{qstn:3-4max}
Let $\phi$ be a quadratic map defined over $\mathbb{Q}$ with a $\mathbb{Q}$-rational periodic critical point of period $2,3$ or $4$. Is it true that $\phi$ has no $\mathbb{Q}$-periodic point of period greater or equal to $3$ that does not belong to a critical cycle?
\end{qstn}



We expect that providing a complete classification of the cases of a $\mathbb{Q}$-rational periodic critical point of period $3$ or $4$ will in effect provide a complete classification of quadratic maps defined over $\mathbb{Q}$ with a $\mathbb{Q}$-rational periodic critical point. 
\begin{qstn}
Let $\phi$ be a endomorphism of $\mathbb{P}^1$ of degree $2$ defined over $\mathbb{Q}$. Is it true that $\phi$ has no $\mathbb{Q}$-rational periodic critical point of period greater or equal to $5$?
\end{qstn}
So far, we have only been able to prove this for the case of a $\mathbb{Q}$-periodic critical point of period $5$
(see Proposition~\ref{prop:critical5} below).

We would like to mention similar types of classifications that have appeared in the literature. Manes~\cite{article:manes2008} gave a full classification of preperiodicity graphs of quadratic rational maps defined over $\mathbb{Q}$ having an automorphism by assuming that such maps have no $\mathbb{Q}$-periodic points of period greater than $4$; she used similar techniques as Poonen~\cite{article:poonen1998} and in turn inspired the style of proofs appearing this paper. The families of quadratic maps having a marked periodic point of period $\leq{6}$ has been studied by Blanc, Canci and Elkies~\cite{article:blanc-canci-elkies2015}. Benedetto et al.~\cite{article:benedetto-et-al2014} gave an explicit equation for the surface parametrizing quadratic maps having a marked preperiodic whose fifth iterate is a periodic point of period $2$. They proved that this surface is an elliptic surface of positive rank over $\mathbb{Q}(t)$. They proposed the conjecture that $\#\PrePer(\phi,\mathbb{Q})\leq{14}$ for $\phi$ a quadratic map of degree $2$ defined over $\mathbb{Q}$ based on experimental data. The classification of quadratic polynomials over quadratic number fields has been studied by Doyle, Faber and Krumm~\cite{article:doyle-faber-krumm2014}. In a recent preprint, Krumm~\cite{article:krumm2015} proved a local-global dynamical criterion for quadratic polynomials.

Finally, we remark that one can generalize the scope of the one-dimensional families of quadratic maps studied even further. Instead of looking at quadratic maps with a periodic critical point, one can take quadratic maps with a marked periodic point with a fixed multiplier $\lambda$. This gives the curves $\Per_n(\lambda)$ parametrizing conjugacy classes of quadratic maps with a marked periodic of period $n$ with multiplier $\lambda$. In this context, Poonen's classification of quadratic polynomials over $\QQ$ is classification of preperiodicity structures in $\Per_1(0)$ and our classification of quadratic maps defined over $\QQ$ with a $\QQ$-rational periodic critical point is classification of preperiodicity structures in $\Per_2(0)$. 
It is worth mentioning that Milnor~\cite{article:milnor1993} proved that $\Per_1(\lambda)$ and $\Per_2(\lambda)$ (for $\lambda\neq{1}$) are embedded as lines in the moduli space $M_2$ of all quadratic maps modulo linear conjugation, while this is not true for $Per_n(\lambda)$ for $n\geq{3}$. This gives another correlation between $\Per_1(0)$ and $\Per_2(0)$ and for our studying the latter.


\subsection*{Acknowledgments}
We would like to thank Joseph H.\ Silverman for suggesting to the first author to look at the curves $\Per_n(\lambda)$. We also thank Fabrizio Barroero, Laura Capuano, Xander Faber, Patrick Ingram, Vincenzo Mantova, Olaf Merkert, Michael Stoll and Umberto Zannier for their comments and suggestions.

\section{Background}
A point $P\in\mathbb{P}^1(\mathbb{C})$ is a \emph{critical point} of an endomorphism $\phi:\mathbb{P}^1\rightarrow\mathbb{P}^1$ defined over $\mathbb{C}$ if $\phi'(P)=0$ (If $P=\infty$ or $\phi(P)=\infty$ we first need to conjugate $\phi$ by a projective automorphism in order to move $P$ and $\phi(P)$ away from infinity). It is not difficult to check that an endomorphism of degree $d$ has $2d-2$ critical points, counted with multiplicity (see Silverman~\cite[\S{1.2}, Theorem 1.1]{book:silverman2007}). This means that a quadratic map has exactly $2$ critical points (a critical point of a quadratic map can have no multiplicity greater than $1$) and is post-critically finite (PCF) if both of these points are preperiodic.

A quadratic map defined over $K$ has a representation as a rational function 
\begin{equation} \label{eq:quadratic-map}
\phi(z)=\frac{a_2z^2+a_1z+a_0}{b_2z^2+b_1z+b_0}=F(z)/G(z)
\end{equation}
with coefficients $a_i,b_i\in{K}$, $0\leq{i}\leq{2}$. To be a proper degree $2$ map, the resultant $Res(F,G)$ must be nonzero. The endomorphism can also be presented in homogeneous coordinates as
\[
\phi: [u:v] \mapsto [a_2u^2+a_1uv+a_0v^2:b_2u^2+b_1uv+b_0v^2] = [F_1(u,v):G_1(u,v)].
\]

For each quadratic map $\phi$ there exists a sequence of polynomials $\left(\Phi_{i,\phi}^*(z)\right)_{i=1}^{\infty}$ called the \emph{dynatomic polynomials} of $\phi$ (see Silverman~\cite[\S{4.2}]{book:silverman2007}). These polynomials have the property that every $P\in \Per(\phi, \bar{K})$ of period $N$  is a root of $\Phi_{N,\phi}^*(z)$. The other implication is not always true; the roots of $\Phi_{N,\phi}^*(z)$ are called \emph{formal periodic points} of period $N$. Formal periodic points that are not periodic of period $N$ are periodic points of period dividing $N$ that have a ``multiplicity". This multiplicity is not directly related to the multiplicity of the roots of $\Phi_{N,\phi}^*(z)$. For example, for the map
\begin{equation*}
\phi(z)=\frac {{z}^{2}+5\,z+3}{{z}^{2}}
\end{equation*}
the point $-1$ is a double root of $\Phi_{1,\phi}^*(z)$, but is not a root of $\Phi_{2,\phi}^*(z)$. On the other hand, for the map
\begin{equation*}
\phi(z)=\frac {{z}^{2}-z+1}{{z}^{2}}
\end{equation*}
the point $1$ is a simple root of $\Phi_{1,\phi}^*(z)$ and a double root of $\Phi_{2,\phi}^*(z)$.

We briefly recall the definition of $\Phi_{n,\phi}^*$. Let $\phi^n(u:v) = [F_n(u,v), G_n(u,v)]$ be the $n$-th iteration of $\phi$ and define
\[
\Phi_{n,\phi}(u,v) = vF_n(u,v) - uG_n(u,v).
\]
We then define the dynatomic polynomials by
\begin{equation*}
\Phi_{n,\phi}^*(u,v) = \prod_{k|n} (\Phi_{k,\phi}(u,v))^{\mu(n/k)},
\end{equation*}
where $\mu$ is the Moebius mu function (the fact that these are actually polynomials was proved by Morton and Silverman~\cite{article:morton-silverman1995}). We will drop the $\phi$ subscript in the notation, since it is usually clear which map we are referring to, and just denote the dynatomic polynomials by $\Phi_n^*$.

The degrees of the dynatomic polynomials provide bounds on the number of periodic points of period $N\geq{1}$. We collect some of this information in the next lemma.

\begin{lem} \label{lem:periodic-limits}
Let $\phi$ be a quadratic map defined over a field $K$. Then $\#\Per_1(\phi, K) \leq{3}$ and \\ $\#\Per_2(\phi, K) \leq{2}$.
\end{lem}

\begin{proof}
This is immediate from $\deg{\Phi_1^*} = 3$ and $\deg{\Phi_2^*} = 2$. 
\end{proof}

\begin{rem} \label{rem:2-fixed}
Suppose the three roots of $\Phi_1^*$ are distinct. If there are two fixed points defined over $\mathbb{Q}$ then the third fixed point is also defined over $\mathbb{Q}$. Therefore the only way a graph $G_\phi$ (over $\mathbb{Q}$) has exactly two fixed points, is if $\Phi_1^*$ has a multiple root.
\end{rem}

\section{Normal forms for realizable graphs} \label{sec:normal}

\begin{prop}
A graph realizing a non-PCF quadratic map defined over a number field $K$ with a $K$-rational periodic critical point has an odd number of vertices (i.e., there is an odd number of preperiodic points defined over $K$).
\end{prop}

\begin{proof}
Let $G_\phi$ be a graph realizing a non-PCF quadratic map $\phi$ with a rational periodic critical point. For each vertex $P$ in the graph $G_\phi$ we have $\deg^{+}(P) = 1$ (the notation $\deg^{+}$ is for the \emph{outdegree} of a vertex of a directed graph, i.e.\ the number of arrows leading out of this vertex) since $P$ corresponds to a preperiodic point which has only one image under $\phi$. For each such $P$ we also have $\deg^{-}(P)\leq{2}$ and we can only have $\deg^{-}(P)=1$ if $P$ is a critical value. Since there is only one critical preperiodic point, we have $\deg^{-}(P)=1$ for exactly one vertex $P$. Therefore we have
\begin{equation*}
\#\PrePer(\phi,K) = \sum_{P\in{G}} 1 = \sum_{P\in{G}} \deg^{+}(P) = \sum_{P\in{G}} \deg^{-}(P) = 2k+1
\end{equation*}
where $k$ is the number of vertices with $\deg^{-}(P) = 2$.
\end{proof}

We denote the following recursively defined (infinite) set of finite directed graphs by $\mathcal{G}$:
\begin{itemize}
\item \emph{Base step:} If $G$ is a finite directed graph consisting of cycle $\mathfrak{C}$ of length $N$, such that for all but one vertices $P\in\mathfrak{C}$ there exists a vertex $Q$ with an arrow from $Q$ to $P$ (i.e., there are $2N-1$ vertices in the graph), then $G\in\mathcal{G}$. In other words, by the base step of the recursion the set $\mathcal{G}$ contains the graphs
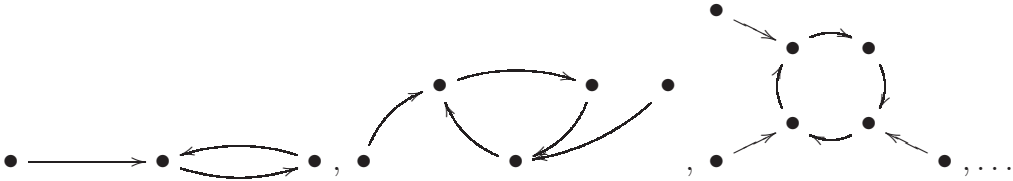
\begin{equation*}
\xygraph{ 
		!{<0cm,0cm>;<1cm,0cm>:<0cm,1cm>::} 
		!{(-2,0) }*+{\bullet}="a" 
		!{(0,0) }*+{\bullet}="b" 
		!{(2,0) }*+{\bullet}="c" 
		"a":"b"
		"b":@/_/"c"
		"c":@/_/"b"
	},
	\xygraph{ 
		!{<0cm,0cm>;<2cm,0cm>:<0cm,1cm>::} 		
		!{(0,1) }*+{\bullet}="a" 
		!{(1,1) }*+{\bullet}="b" 
		!{(0.5,0) }*+{\bullet}="c" 
		!{(-0.5,0) }*+{\bullet}="d"
		!{(1.5,1	) }*+{\bullet}="e"		
		"a":@/^/"b"
    		"b":@/^/"c"
      		"c":@/^/"a"
    		"d":@/^/"a"
      		"e":@/^/"c"       
	},
	\xygraph{ 
		!{<0cm,0cm>;<1cm,0cm>:<0cm,1cm>::} 
		!{(0,1.5) }*+{\bullet}="a" 
		!{(1,1.5) }*+{\bullet}="b" 
		!{(1,0.5) }*+{\bullet}="c" 
		!{(0,0.5) }*+{\bullet}="d"
		!{(-1,2) }*+{\bullet}="e"
		!{(2,0	) }*+{\bullet}="f"
		!{(-1,0) }*+{\bullet}="g"		
		"a":@/^/"b"
    		"b":@/^/"c"
      		"c":@/^/"d"
       		"d":@/^/"a"     		
    		"e":"a"
      		"f":"c"
      		"g":"d"
	},\ldots
\end{equation*}
\item \emph{Recursive step 1:} If $G\in\mathcal{G}$, let $G'$ be a graph defined by adding to $G$ a new cycle $\mathfrak{C}'$ and for each $P'\in\mathfrak{C}'$ adding a vertex $Q'$ with an arrow from $Q'$ to $P'$; then $G'\in\mathcal{G}$.
\item \emph{Recursive step 2:} Let $G\in\mathcal{G}$ and let $P''$ be a point not on a cycle and such that $\deg^{-}(P'') = 0$; let $G''$ be the graph constructed from $G$ by adding two vertices $Q_1,Q_2$ and arrows from $Q_1$ and $Q_2$ to $P''$. Then $G''\in\mathcal{G}$.
\end{itemize}

\begin{lem} \label{lem:recursive}
If $G_\phi$ is the preperiodicity graph of a non-PCF quadratic map $\phi$ defined over a number field $K$ with a $K$-rational periodic critical point then $G\in\mathcal{G}$.
\end{lem}

\begin{proof}
Let $\mathcal{G}'$ be the subset of $\mathcal{G}$ of graphs $G$ such that $G\subseteq{G_\phi}$. The set $\mathcal{G}'$ is nonempty, since $G_\phi$ must contain a critical cycle, and this critical cycle together with the preimages of non-critical values on this cycle is an element of $\mathcal{G}$ by the base step of the definition of $\mathcal{G}$. The set $\mathcal{G}'$ is partially ordered by isomorphic subgraph partial ordering, and it is finite because the set of subgraphs of any finite directed graph is finite. Therefore there exists a maximal element $G'\in\mathcal{G}'$. We prove that $G'=G_\phi$:
\begin{itemize}
\item There cannot exist a vertex $P\in{G'}$ with in-degree $0$ in $G'$ that has in-degree $2$ in $G_\phi$, since otherwise we can add the two preimages of $P$ to $G'$ and construct a graph in $\mathcal{G}$ (using recursion step $2$) that strictly contains $G'$, contradicting the maximality of $G'$. 
\item There cannot exist a vertex $P\in{G'}$ with in-degree $0$ in $G'$ that has in-degree $1$ in $G$, since the only vertex with in-degree $1$ in $G_\phi$ is on the critical cycle, and its preimage is in $G'$.
\item There is no path from a vertex of $G_\phi$ that is not in $G'$ that ends in a vertex of $G'$. This is true from the last two items, since the first vertex of $G'$ in the path will have to satisfy one of the conditions described in the two items, an impossibility.
\item Since the out-degree of each vertex in $G_\phi$ is $1$, if we start from any vertex $P$ of $G_\phi$ not in $G'$, any maximal path starting from $P$ must contain a cycle. Adding the cycle and its direct preimages to $G'$ creates a subgraph of $G_\phi$ that strictly contains $G'$ and is in $\mathcal{G}$ by recursion step 1 in the definition of $\mathcal{G}$, contradicting the maximality of $G'$. 
\end{itemize}
We have thus proved that there is no vertex of $G_\phi$ that is not in $G'$. However, from this we know that $G'$ and $G_\phi$ must be equal, since both have the property that the outdegree of each vertex is $1$. 
\end{proof}

%

We now present some normal forms for a quadratic map with a marked periodic critical point that will help us in the sequel. By \emph{normal form} corresponding to a graph $G$, we mean a parametrization $t\mapsto\phi_t$ (parameter name may vary) such that for all but finitely many values of parameter $t$ the quadratic map $\phi_t$ admits the graph $G$. Moreover, for each conjugacy class of a quadratic map admitting the graph $G$ there exists a unique value of $t$ such that $\phi_t$ is in the conjugacy class.

\begin{lem} \label{lem:cyc2}
A quadratic map with a periodic critical point of period $2$ is linearly conjugate to a map of the form
\begin{equation} \label{eq:cyc2}
\phi(z)= \frac{a_1z+1}{b_2z^2},
\end{equation}
having $0$ as a periodic critical point of period $2$ and $\infty$ as the critical value of $0$. The graph $G_\phi$ contains the graph
\begin{equation*}
\xygraph{ 
		!{<0cm,0cm>;<1cm,0cm>:<0cm,1cm>::} 
		!{(-2,0) }*+{\bullet_{-1/a_1}}="a" 
		!{(0,0) }*+{\bullet_{0}}="b" 
		!{(2,0) }*+{\bullet_{\infty}}="c" 
		"a":"b"
		"b":@/_/"c"
		"c":@/_/"b"
	}
\end{equation*}
\end{lem}

\begin{proof}
Using the well known properties of the automorphisms in $\PGL_2$, we can replace any three points $P,Q,R$ in the preperiodicity graph of $\phi$ with the points $0, \infty,1$ by conjugating by the appropriate element of $\PGL_2(\mathbb{Q})$. So we can replace the periodic critical point with $0$ and its image with $\infty$. Since $\phi(0)=\infty$ we get that $b_0 = 0$ in the presentation~\eqref{eq:quadratic-map} of $\phi$, and $\phi(\infty)=0$ implies $a_2=0$. For $0$ to be a critical point we must have $b_1=0$, so that the only preimage of $\infty$ is $0$. Finally, we know $a_0\neq{0}$, since otherwise the map $\phi$ degenerates, and since we can always scale the coefficients by a constant, we can scale them so that $a_0=1$. Checking that the second preimage of $0$ is $-1/a_1$ is done by solving the equation $\phi(\alpha) = 0$.
\end{proof}

Any quadratic map with a periodic critical point of period $2$ has infinitely many representatives of the form \eqref{eq:cyc2} since we have only chosen the images of $0$ and $\infty$. Therefore this ``normal form" is not canonical, but the extra freedom of choice gives flexibility in the inadmissibility proofs (see Section~\ref{sec:inadmissible}) of some of the graphs in Table~\ref{table:n2e}. We can  remove this ambiguity by replacing the point $-1/a_1$ with $1$, as in the following lemma.

\begin{lem} \label{lem:cyc2-unique}
A normal form for a quadratic map with a periodic critical point of period $2$ (i.e.\ admitting graph R2P1) is
\begin{equation} \label{eq:cyc2-unique}
\phi(z)= \frac{-z+1}{b_2z^2},
\end{equation}
having $0$ as the periodic critical point of period $2$, $\infty$ as the critical value of $0$ and $1$ as the second preimage of $0$. The graph $G_\phi$ contains the graph
\begin{equation*}
\xygraph{ 
		!{<0cm,0cm>;<1cm,0cm>:<0cm,1cm>::} 
		!{(-2,0) }*+{\bullet_{1}}="a" 
		!{(0,0) }*+{\bullet_{0}}="b" 
		!{(2,0) }*+{\bullet_{\infty}}="c" 
		"a":"b"
		"b":@/_/"c"
		"c":@/_/"b"
	}
\end{equation*}
\end{lem}

\begin{lem} \label{lem:cyc2-1}
A normal form of a quadratic map having a periodic critical point of period $2$ and a fixed point (i.e.\ admitting graph R2P2) is
\begin{equation} \label{eq:cyc2-1}
\phi(z)= \frac{a_1z+1}{(a_1+1)z^2},
\end{equation}
having $0$ as a periodic critical point of period $2$, $\infty$ as the critical value of $0$ and $1$ as a fixed point. The graph $G_\phi$ contains the graph
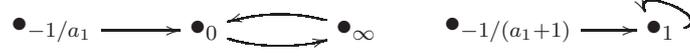
\begin{equation*}
\xygraph{ 
		!{<0cm,0cm>;<1cm,0cm>:<0cm,1cm>::} 
		!{(-2,0) }*+{\bullet_{-1/a_1}}="a" 
		!{(0,0) }*+{\bullet_{0}}="b" 
		!{(2,0) }*+{\bullet_{\infty}}="c" 
		!{(4,0) }*+{\bullet_{-1/(a_1+1)}}="d"
		!{(6,0) }*+{\bullet_{1}}="e" 		 
		"a":"b"
		"b":@/_/"c"
		"c":@/_/"b"
		"d":"e"
		"e":@(r,lu) "e"
	}
\end{equation*}
\end{lem}
\begin{proof}
This is proved by using Lemma~\ref{lem:cyc2}, replacing the fixed point with $1$ and finding its other preimage by solving $\phi(\alpha)=1$. 
\end{proof}

\begin{prop} \label{prop:R2P3}
A normal form of a quadratic map admitting graph R2P3 in Table~\ref{table:r2} is
\begin{equation*}
\phi(z)=-\frac{a_1z+1}{a_1(a_1+1)z^2}.
\end{equation*}
\end{prop}

\begin{proof}
We start with the normal form~\eqref{eq:cyc2} in Lemma~\ref{lem:cyc2} and add the constraint $\phi(1)=-1/a_1$. From this we get the condition $a_1^2+a_1+b_2=0$, or $b_2=-(a_1+1)a_1$.
\end{proof}

\begin{prop}\label{prop:R2P4}
There is a unique conjugacy class of quadratic maps defined over $\mathbb{Q}$ with a $\mathbb{Q}$-rational periodic critical point of period $2$ that realizes graph R2P4.
\end{prop}

\begin{proof}
We mentioned in Remark~\ref{rem:2-fixed} that $\Phi_1^*$ must have a double root. We start with the normal form \eqref{eq:cyc2-1} in Lemma~\ref{lem:cyc2-1}. We add the constraint that $1$ is a double root of $\Phi_1^*$. This is done by constraining $1$ to be a root of the derivative of $\Phi_1^*$. This condition gives $ 2a_1+3 = 0 $ or $a_1=-\frac{3}{2}$. This provides the example in Table~\ref{table:r2} for graph R2P4.
\end{proof}

\begin{prop} \label{prop:R2P5}
A normal form of a quadratic map admitting graph R2P5 in Table~\ref{table:r2} is
\begin{equation*}
\phi(z)=-\frac{a_1(a_1z+1)}{(a_1^3+2a_1^2+2a_1+1)z^2}.
\end{equation*}
\end{prop}

\begin{proof}
We start with the normal form~\eqref{eq:cyc2} in Lemma~\ref{lem:cyc2} and add the constraint $\phi^2(1)=-1/a_1$. From this we get the condition $a_1^3+2a_1^2+b_2a_1+2a_1+1=0$, or $b_2=-\frac{a_1^3+2a_1^2+2a_1+1}{a_1}$.
\end{proof}

\begin{prop} \label{prop:R2P6}
A normal form of a quadratic map admitting graph R2P6 in Table~\ref{table:r2} is
\begin{equation*}
\phi(z)=\frac{(w^2+1)z-w}{(w^2-w+1)z^2}.
\end{equation*}
\end{prop}

\begin{proof}
We start with the normal form~\eqref{eq:cyc2-1} in Lemma~\ref{lem:cyc2-1} and add the constraint $\phi(w)=-1/(a_1+1)$; this makes $w$ a preimage of $-1/(a_1+1)$, which is in turn a preimage of the fixed point $1$. From this we get the condition $a_1w+1+w^2$, or $a_1=-\frac{1+w^2}{w}$. Substituting this back into equation \eqref{eq:cyc2-1} gives us the required result.
\end{proof}

\begin{prop} \label{prop:R2P7}
A normal form of a quadratic map admitting graph R2P7 in Table~\ref{table:r2} is
\begin{equation*}
\phi(z)=\frac{(w^2+w+1)z-w(w+1)}{z^2}.
\end{equation*}
\end{prop}

\begin{proof}
We construct a normal form for graph~R2P4; this is the normal form in the statement of the Proposition. By Proposition~\ref{prop:R2P4} there is a unique parameter $w$ for which this normal form realizes graph~R2P4. By Remark~\ref{rem:2-fixed} any other parameter $w$ for which the normal form does not degenerate, must admit graph~R2P7. We start with the normal form~\eqref{eq:cyc2-1} in Lemma~\ref{lem:cyc2-1}. We know the first dynatomic polynomial $\Phi_1^*$ is divisible by $(z-1)$, since $1$ is a fixed point. To constrain the existence of another fixed point, we add the condition that $w$ is a root of the polynomial $\frac{\Phi_1^*(z)}{z-1}$. This gives us $a_1w^2+w^2+a_1w+w+1=0$ or $a_1=-{\frac {{w}^{2}+w+1}{w \left( w + 1\right) }}$, and from this we obtain the required normal form.
\end{proof}

From Lemmas~\ref{lem:cyc2-unique} and~\ref{lem:cyc2-1} and Propositions~\ref{prop:R2P3}, \ref{prop:R2P4}, \ref{prop:R2P5}, \ref{prop:R2P6} and \ref{prop:R2P7} we immediately get the following corollary.
\begin{cor} \label{cor:genus0}
The curves parametrizing (conjugacy classes of) quadratic maps admitting the graphs in Table~\ref{table:r2} are rational over $\mathbb{Q}$, i.e.\ of genus $0$.
\end{cor}

%

\section{Inadmissibility results} \label{sec:inadmissible}

\begin{prop} \label{prop:inadmissible}
The graphs N2E1--N2E5 in Table~\ref{table:n2e} are inadmissible by non-PCF quadratic maps defined over $\mathbb{Q}$ with a $\mathbb{Q}$-rational periodic critical point of period $2$ defined over $\mathbb{Q}$. 
\end{prop}

\begin{proof} \quad
\begin{description}[leftmargin=0cm, labelindent=0cm]
\item[N2E1] We start with the normal form \eqref{eq:cyc2-1} in Lemma~\ref{lem:cyc2-1}, and add the constraint $\phi(z)=-\frac{1}{a_1}$. This gives the equation
\begin{equation} \label{eq:N2E1}
a_1^2z+a_1+z^2a_1+z^2=0.
\end{equation}
Each $\phi$ realizing graph~N2E1 provides two solutions on the affine curve defined by \eqref{eq:N2E1}, one for each preimage of $-\frac{1}{a_1}$.
A birational transformation defined by $$x=-\frac{1}{a_1}, \quad y=-1-\frac{1}{a_1z}$$ brings equation~\eqref{eq:N2E1} to the form
\begin{equation} \label{eq:N2E1-ell}
y^2 + y = x^3 - x^2.
\end{equation} 
This is elliptic curve 11a3 in Cremona's tables~\cite{article:cremona2006}, with rank $0$ and torsion subgroup $\mathbb{Z}/5\mathbb{Z}$ with rational points
$$\{\mathcal{O}, (0,0), (0,-1), (1,-1), (1,0)\}.$$ 
Since both equations \eqref{eq:N2E1} and \eqref{eq:N2E1-ell} correspond to smooth projective curves, the map between the two curves is in fact an isomorphism, and the five rational points of \eqref{eq:N2E1-ell} are mapped to the five unique solutions of \eqref{eq:N2E1}:
$$[0 : 0 : 1], [-1 : 1 : 1], [-1 : 1 : 0], [0 : 1 : 0], [1 : 0 : 0],$$
where the coordinates are $[X:Y:Z]$ with $a_1=X/Z$ and $z=Y/Z$. The first point has $z=0$, which is not a preimage of $-\frac{1}{a_1}$, the next two points correspond to degenerate maps, i.e.\ with resultant $0$, and the final two points are at infinity. The points at infinity are degenerate limit points of quadratic maps (when $Z$ tends to $0$ in the parameter $a_1=X/Z$ the normal form \eqref{eq:cyc2-1} degenerates, i.e.\ drops in degree).Thus we see that there are no quadratic maps defined over $\mathbb{Q}$ realizing graph~N2E1.
\begin{rem*}
In all the other cases, we will consider the normal forms from the previous section together with some other constraints,  as we did for N2E1. We will obtain affine curves whose projective closure admits some points at infinity. As in the previous case, the set of points at infinity can be verified to correspond to degenerate quadratic maps, i.e.\ of degree $0$ or $1$. 
\end{rem*}
\item[N2E2] We start with normal form \eqref{eq:cyc2} in Lemma~\ref{lem:cyc2}, and add the constraint $\phi^3(1)=-\frac{1}{a_1}$. This gives us the equation of our curve $C$:
\begin{tiny}
\begin{equation}
C: a_1^6+4a_1^5+7a_1^4+7a_1^3+2a_1^4b_2+4a_1^3b_2+2a_1^2b_2+4a_1^2+a_1+2a_1^2b_2^2+2a_1b_2^2+b_2^3=0.
\end{equation}
\end{tiny}
The projective closure $\bar{C}$ of $C$ is given by
\begin{tiny}
\begin{equation*}
\bar{C}: X^6 + 4X^5Z + 2X^4YZ + 7X^4Z^2 + 4X^3YZ^2 + 2X^2Y^2Z^2 + 7X^3Z^3 + 2X^2YZ^3 + 2XY^2Z^3 + Y^3Z^3 + 4X^2Z^4 + XZ^5=0,
\end{equation*}
\end{tiny}
with coordinates $[X:Y:Z]$ where $a_1=X/Z$ and $b_2=Y/Z$.

Using \magma~\cite{article:magma} we find that this is a genus $1$ curve, and a quick search provides us with three rational points:
$$ [0 : 1 : 0], [0 : 0 : 1], [-1 : 0 : 1].$$
The point $[0:0:1]$ is nonsingular, so we may perform a Weierstrass transformation to the elliptic curve obtained by choosing $[0:0:1]$ as the identity element of the curve $\bar{C}$.
Using \magma we find that the minimal model of this elliptic curve is
\begin{equation} \label{eq:N2E2-ell}
E: y^2 + y = x^3 + x^2 + x,
\end{equation}
with projective closure 
\begin{equation*}
\bar{E}: -X^3 - X^2Z + Y^2Z - XZ^2 + YZ^2=0
\end{equation*}
in the projective plane $\mathbb{P}^2_{[X:Y:Z]}$.

The birational map from $\bar{C}$ to $\bar{E}$ is defined by:
\begin{tiny}
\begin{equation*}
\begin{split}
\varphi = [&X^4Y + 3X^3YZ + 2X^2Y^2Z + 4X^2YZ^2 + 2XY^2Z^2 + Y^3Z^2 + 3XYZ^3 + YZ^4: \\
&X^5 + 4X^4Z + 2X^3YZ + 7X^3Z^2 + 4X^2YZ^2 + 2XY^2Z^2 - Y^3Z^2 + 7X^2Z^3 + 2XYZ^3 + 2Y^2Z^3 + 4XZ^4 + Z^5:\\
&Y^3Z^2].
\end{split}
\end{equation*}
\end{tiny}
The inverse is given by:
\begin{tiny}
\begin{equation*}
\begin{split}
\psi = [&X^3Y^2 + X^3YZ + X^2Y^2Z - Y^4Z + X^2YZ^2 + XY^2Z^2 - 2Y^3Z^2 + XYZ^3 - Y^2Z^3 - YZ^4 - 
    Z^5: \\
&XYZ^3 + YZ^4: \\
&Y^2Z^3 + 2YZ^4 + Z^5].
\end{split}
\end{equation*}
\end{tiny}
The elliptic curve $E$ is Cremona 19a3, with Mordell--Weil rank $0$, and torsion subgroup $\mathbb{Z}/3\mathbb{Z}$, easily found to consist of the three rational points
$$[0 : 0 : 1], [0 : -1 : 1], \mathcal{O}=[0:1:0].$$
One can check that outside the sets
\begin{tiny}
\begin{equation*}
\begin{split}
S &= \{[0 : 0 : 1], [-1 : 0 : 1],  [0 : 1 : 0], [-\frac{1}{2}\pm\frac{1}{2}i\sqrt{3} : 0 : 1]\} \subset C(\mathbb{C}), \\
T &= \{[0:1:0],[0:-1:1],[0:0:1], [-\frac{1}{2}\pm\frac{1}{2}i\sqrt{3}: 0: 1],[-\frac{1}{2}\pm\frac{1}{2}i\sqrt{3}:-1:1],[-1:-\frac{1}{2}\pm\frac{1}{2}i\sqrt{3}:1]\} \subset E(\mathbb{C})
\end{split}
\end{equation*}
\end{tiny}
the maps $\phi$ and $\psi$ are isomorphisms. The elliptic curve $E$ has no $\mathbb{Q}$-rational points outside of $T$, so that $C$ has no $\mathbb{Q}$-rational points outside of $S$, meaning that the three $\mathbb{Q}$-rational points in $S$ are the only ones on $C$. One of these is at infinity, and the two others correspond to degenerate maps with resultant $0$. 
\item[N2E3] We start with normal form \eqref{eq:cyc2-1} in Lemma~\ref{lem:cyc2-1}, and add the constraint $\phi^2(z)=-\frac{1}{a_1+1}$. This gives an equation for a plane curve $C$:
\begin{equation} \label{eq:N2E3}
C: z^3a_1^3+z^3a_1^2+2z^2a_1^2+z^2a_1+z^4a_1^2+2z^4a_1+z^4+2a_1z+1=0.
\end{equation} 
Searching for $\mathbb{Q}$-rational points on the curve $C$ gives us five points:
\begin{equation} \label{eq:N2E3-pts}
[-2:1:1],[-1:1:1],[0:1:0],[-1:1:0],[1:0:0],
\end{equation}
where the coordinates are $[X:Y:Z]$ with $a_1=X/Z$ and $z=Y/Z$.

\magma tells us this curve is of genus $1$. The point $[-2:1:1]$ is nonsingular, and choosing it as the identity element of the curve, we can bring the curve to minimal Weierstrass form:
\begin{equation} \label{eq:N2E3-ell}
E: y^2 + y = x^3 + x^2 + x.
\end{equation}
Surprisingly, we get the same elliptic curve 19a3 from graph N2E2! We quickly repeat the calculations showing that the five points in \eqref{eq:N2E3-pts} are the only rational points of the curve $C$. The birational map from $\bar{C}$ to $\bar{E}$ (where $\bar{C}$ and $\bar{E}$ are the projective closures of $C$ and $E$ in $\mathbb{P}^2_{[X:Y:Z]}$, respectively, with $a_1=X/Z$ and $z=Y/Z$ for $C$) is 
\begin{tiny}
\begin{equation*}
\begin{split}
\varphi = [&
X^2Y^3 + XY^4 - X^2Y^2Z - XY^3Z + Y^4Z + XY^2Z^2 - 2Y^3Z^2 - XYZ^3 + 2Y^2Z^3 - YZ^4: \\
&X^2Y^2Z + XY^3Z - XY^2Z^2 + 2XYZ^3 + 2Y^2Z^3 - 3YZ^4 + 2Z^5: \\
&Y^3Z^2 - 3Y^2Z^3 + 3YZ^4 - Z^5
]
\end{split}
\end{equation*}
\end{tiny}
with inverse
\begin{tiny}
\begin{equation*}
\begin{split}
\psi = [&
-256X^4 + 256X^3Y - 512X^3Z + 768X^2YZ - 256XY^2Z - 256Y^3Z - 256X^2Z^2 + 512XYZ^2 - 512Y^2Z^2: \\
&256X^4 + 512X^3Z + 256X^2Z^2: \\
&256X^4 + 512X^3Z - 256X^2YZ + 256X^2Z^2 - 256XYZ^2
].
\end{split}
\end{equation*}
\end{tiny}
Outside the sets 
\begin{tiny}
\begin{equation*}
\begin{split}
S &= \{[-2:1:1],[-1 : 1 : 1],  [0 : 1 : 0], [-1 : 1 : 0], [1:0:0]\} \subset C(\mathbb{C}), \\
T &= \{[0:1:0],[0:-1:1],[0:0:1], [-\frac{1}{2}\pm\frac{1}{2}i\sqrt{3}: 0: 1],[-\frac{1}{2}\pm\frac{1}{2}i\sqrt{3}:-1:1],[-1:-\frac{1}{2}\pm\frac{1}{2}i\sqrt{3}:1]\} \subset E(\mathbb{C})
\end{split}
\end{equation*}
\end{tiny}
the birational maps $\varphi$ and $\psi$ are isomorphisms. The elliptic curve $E$ has no $\mathbb{Q}$-rational points outside of $T$, so that $C$ has no $\mathbb{Q}$-rational points outside of $S$, meaning that the five $\mathbb{Q}$-rational points in $S$ are the only ones on $C$. These five points do not correspond to maps realizing graph N2E3, since they are either at infinity or have $z=1$, clearly impossible, since $1$ is periodic, while $z$ should be strictly preperiodic.
\item[N2E4] We start with normal form \eqref{eq:cyc2-1} in Lemma~\ref{lem:cyc2-1}, and add two constraints: First, we require a $\mathbb{Q}$-rational root for $\frac{\Phi_1^*(z)}{z-1}$ (by Proposition~\ref{prop:R2P4} one of these must be different from $1$). Second, we add the constraint $\phi(w)=-\frac{1}{a_1+1}$. We get the system:
\begin{eqnarray*}
\begin{cases}
z^2a_1+z^2+a_1z+z+1 = 0, \\ 
w^2+a_1w+1 = 0.
\end{cases}
\end{eqnarray*}
Eliminating the variable $a_1$ we get a curve $C$ defined by the equation
\begin{equation} \label{eq:N2E4}
C: -z^2-z^2w^2+z^2w-z-zw^2+zw+w=0.
\end{equation}
The curve $C$ is nonsingular of genus $1$, and a search provides us with four $\mathbb{Q}$-rational points:
$$  [0 : 1 : 0], [0 : 0 : 1], [-1 : 0 : 1], [1 : 0 : 0], $$
where the coordinates are $[X:Y:Z]$ with $z=X/Z$ and $w=Y/Z$.

Choosing $[0:0:1]$ as the identity element, and bringing to minimal Weierstrass form gives us the elliptic curve
\begin{equation*}
E: y^2 + xy + y = x^3 + x^2.
\end{equation*}
This is curve 15a8 in Cremona's database, which has Mordell--Weil rank $0$ and torsion subgroup $\mathbb{Z}/4\mathbb{Z}$, where the four torsion points are:
$$  [0 : 0 : 1], [0 : -1 : 1], [-1 : 0 : 1],  \mathcal{O} = [0 : 1 : 0].$$
The curves $C$ and $E$ are isomorphic by the map
$$ x = z-{\frac {z}{w}}+{\frac {z}{{w}^{2}}}+{w}^{-2}, \quad
y = -z-1-{\frac {z}{{w}^{3}}}-{w}^{-2}-{w}^{-3}$$ 
with inverse
$$ z = {\frac {y \left( y+1 \right) ^{2}}{{y}^{3}+{y}^{2}+yx+y-{x}^{2}}}, \quad
w ={\frac { \left( y+1 \right)  \left( -yx-y+{x}^{2} \right) }{{y}^{3}+
{y}^{2}+yx+y-{x}^{2}}}.$$
Since the curves are isomorphic, we see that the four $\mathbb{Q}$-rational points we found on $C$ are the only $\mathbb{Q}$-rational points. None of these points correspond to maps realizing graph N2E4, since they are either at infinity, or have either $z=0$ or $w=0$, which is clearly impossible. 
\item[N2E5] We start with normal form \eqref{eq:cyc2} in Lemma~\ref{lem:cyc2}, and add the constraint that $1$ is a preimage of $-\frac{1}{a_1}$. We get
\begin{equation*}
b_2 = -a_1(1+a_1).
\end{equation*}
The other preimage of $-\frac{1}{a_1}$ has to be
$$-\frac{1}{a_1+1},$$
so we add two constraints, that $\phi(z)=1$ and $\phi(y)=-\frac{1}{a_1+1}$. These give the following two equations:
\begin{eqnarray*}
\begin{cases}
-a_1z-1-z^2a_1^2-z^2a_1 &= 0,\\
-a_1y-1+y^2a_1 &= 0.
\end{cases}
\end{eqnarray*}
Using the second equation, we eliminate $a_1$ and get a single equation:
\begin{equation*}
C : (-1+y-y^2)z^2+(-y^2+y)z-y^4-y^2+2y^3 = 0.
\end{equation*}
One can check using \magma that the curve $C$ is birational to $C' : X^2+Y^2+Z^2 = 0$,
and therefore $C$ can only have singular rational points. These are:
\[[0 : 0 : 1], [1 : 0 : 1], [0 : 1 : 0],\]
where the coordinates are $[X:Y:Z]$ with $y=X/Z$ and $z=Y/Z$.
None of these points correspond to quadratic maps realizing N2E5: one of the points is at infinity, and the others either have $y=0$ or $z=0$, and this is impossible, since $0$ is periodic, while $y$ and $z$ should be strictly preperiodic. 
\end{description}
\end{proof}

The next proposition implies Theorem~\ref{thm:main-thm1}.

\begin{prop} \label{prop:hasse}
Assume Conjecture~\ref{conj:main2}. Any non-PCF quadratic map with a periodic critical point of period $2$ that does not realize one of the graphs in Table~\ref{table:r2} must admit one of the graphs in Table~\ref{table:n2e}.
\end{prop}

\begin{proof}
We can represent the graphs in Tables~\ref{table:n2e} and \ref{table:r2} in a Hasse diagram (a graph representation of a partially ordered set, see for instance K.\ Rosen~\cite[\S{1.4.3}]{book:rosen1999}) with respect to the partial order of subgraph isomorphism:
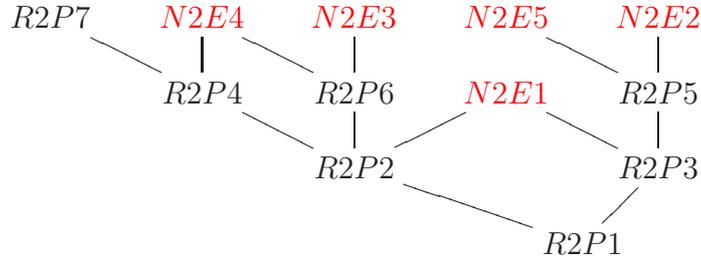
\begin{equation*}
\xygraph{
!{<0cm,0cm>;<1cm,0cm>:<0cm,1cm>::}
!{(0,0) }*+{{R2P1}}="R2P1"
!{(-3,1) }*+{{R2P2}}="R2P2"
!{(1,1) }*+{{R2P3}}="R2P3"
!{(-5,2) }*+{{R2P4}}="R2P4"
!{(-3,2) }*+{{R2P6}}="R2P6"
!{(1,2) }*+{{R2P5}}="R2P5"
!{(-1,2) }*+{{\color{red}N2E1}}="N2E1"
!{(-7,3) }*+{{R2P7}}="R2P7"
!{(-5,3) }*+{{\color{red}N2E4}}="N2E4"
!{(-3,3) }*+{{\color{red}N2E3}}="N2E3"
!{(1,3) }*+{{\color{red}N2E2}}="N2E2"
!{(-1,3) }*+{{\color{red}N2E5}}="N2E5"
"R2P1"-"R2P2"
"R2P1"-"R2P3"
"R2P2"-"R2P4"
"R2P2"-"R2P6"
"R2P2"-"N2E1"
"R2P3"-"R2P5"
"R2P3"-"N2E1"
"R2P4"-"R2P7"
"R2P4"-"N2E4"
"R2P6"-"N2E3"
"R2P6"-"N2E4"
"R2P5"-"N2E2"
"R2P5"-"N2E5"
}
\end{equation*}

This Hasse diagram is complete with respect to the recursively defined set $\mathcal{G}$ introduced in Section~\ref{sec:normal} in the sense that for any graph $G$ in the diagram, if we  use recursion step 1 or 2 (in the definition of the set $\mathcal{G}$ in Section~\ref{sec:normal}) to produce a new graph $G'$ from $G$ then $G'$ is either already in the list, or contains one of the graphs $N2E1,...,N2E5$. Note that according to Conjecture~\ref{conj:main2} and Lemma~\ref{lem:periodic-limits}, we can only add 1-cycles (recursion step 1) to a graph, and we can only add at most three of these (by Lemma~\ref{lem:periodic-limits}).

We show this explicitly for the graph R2P7: If we use step 2 to add two preimages to the preimage of $0$, then we get a graph containing the graph N2E1. If we use step 2 to add two preimages to the preimage of one of the fixed points, we get a graph containing the graph N2E4. We cannot use step 1 since there are already three fixed points.

Now suppose $\phi$ is a non-PCF quadratic map defined over $\mathbb{Q}$ with a $\mathbb{Q}$-rational periodic critical point of period $2$ such that $G_\phi$ is not on the list. By Lemma~\ref{lem:recursive} we know that $G_\phi$ is in $\mathcal{G}$, and therefore can be obtained from R2P1 by a finite number of recursion steps, and therefore must admit one of the graphs N2E1,...,N2E5. But these graphs are inadmissible, so that the graph $G_\phi$ is also inadmissible, contradiction.
\end{proof}

\begin{thm}
A quadratic map defined over $\QQ$ with a $\QQ$-rational period critical point of period $2$ has no $\QQ$-rational periodic point of period $N$ for $3\leq{N}\leq{6}$.
\end{thm}

\begin{proof} \quad
\begin{description}[leftmargin=0cm, labelindent=0cm]
\item[N=3] We start with normal form \eqref{eq:cyc2} in Lemma~\ref{lem:cyc2}, and add the constraint that $1$ is a root of the dynatomic polynomial $\Phi_3^*(z)$. We get the equation
\begin{equation} \label{eq:N=3}
b_2^2+(2a_1^2+3a_1+1)b_2+a_1^4+3a_1^3+4a_1^2+3a_1+1=0.
\end{equation}
This equation is reducible over $\mathbb{Q}(\sqrt{-3})$; we get:
\begin{equation*}
b_2 = -(a_1^2+\frac{3}{2}a_1+\frac{1}{2})\pm\frac{\sqrt{-3}}{2}(a_1+1).
\end{equation*}
For any $a_1\neq{-1}$ rational $b_2$ is irrational, so that the only rational point of these two conics not at infinity is $(a_1,b_2)=(-1,0)$ (this is the point of intersection of the two conics). There is another rational point at infinity. 
These two points are the only $\mathbb{Q}$-rational solutions to equation~\ref{eq:N=3}. The point $(-1,0)$ corresponds to a map with resultant $0$. Thus we see that there are no quadratic maps defined over $\mathbb{Q}$ having a $\QQ$-rational periodic critical point of period $2$ and a $\QQ$-rational periodic point of period $3$.
\item[N=4] We start with normal form \eqref{eq:cyc2} in Lemma~\ref{lem:cyc2}, and add the constraint that $1$ is a root of the dynatomic polynomial $\Phi_4^*(z)$. We get the equation
\begin{equation} \label{eq:N=4}
\begin{split}
b_2^4&+(4a_1^2+1+5a_1)b_2^3+ \\
&(18a_1^3+6a_1+16a_1^2+1+7a_1^4)b_2^2+\\
&(25a_1^2+25a_1^5+43a_1^3+1+8a_1+44a_1^4+6a_1^6)b_2+\\
&1+7a_1+2a_1^8+12a_1^7+55a_1^5+\\
&61a_1^4+33a_1^6+46a_1^3+23a_1^2=0.
\end{split}
\end{equation}
After resolution of singularities, we find that the curve $C$ defined by equation \eqref{eq:N=4} is birational to
\begin{equation*}
C' : X^2+Y^2+Z^2 = 0,
\end{equation*}
embedded in the projective space $\mathbb{P}^2_{[X:Y:Z]}$. The curve $C'$ has no $\mathbb{Q}$-rational points, and so $C$ cannot have any nonsingular $\mathbb{Q}$-rational points. The singular points of $C$ are 
$$[-1 : 0 : 1], [0 : 1 : 0]$$
(where the coordinates are $[X:Y:Z]$ with $a_1=X/Z$ and $b_2=Y/Z$), one of which is at infinity, and the other one corresponds to a map with $0$ resultant. defined over $\mathbb{Q}$ having a $\QQ$-rational periodic critical point of period $2$ and a $\QQ$-rational periodic point of period $4$. 
\item[N=5] We start with normal form~\eqref{eq:cyc2-unique} in Lemma~\ref{lem:cyc2-unique}. If $z$ is a point of period $5$ then it is a root of the $5$-th dynatomic polynomial. Let $C_1$ be the curve defined by the equation $\Phi_5^*(b_2,z)=0$ in the affine plane $\AA^2_{[b_2,z]}$. Using \magma, one can check that this curve is absolutely irreducible and has genus $4$, and is not hyperlliptic. The curve $C$ has an automorphism $\sigma$ of order $5$ defined by sending a point $(b_2,z)\in{C}$ to $(b_2, \phi(z))$. We now calculate the quotient of $C$ by the cyclic group $\left<\sigma\right>$. We follow the trace map method described in Morton~\cite{article:morton1998}: Let $tr(z)=z+\phi(z)+\phi^2(z)+...+\phi^4(z)$ be the \emph{trace} of the periodic point $z$; the image of the map $\Psi: (b_2,z) \mapsto (b_2, tr(z))$ is birational to the quotient curve $C/\left<\sigma\right>$. When $\phi$ is a polynomial, the image of $\Psi$ can be calculated by taking the resultant of $\Phi_5^*(a_1,z)$ and $t-tr(z)$ with respect to $z$. However, in our case $t-tr(z)$ is not a polynomial, since $tr(z)$ is a rational function in $z$. We fix this by clearing out the denominator of $tr(z)$ and show that this does not add any new roots. Let $\phi^n(z)=f_n(z)/g_n(z)$ where $f_n,g_n\in\QQ[z]$ with no common factor, then the denominator can be cleared by multiplying $t-tr(z)$ by $A(z) := \prod_{n=1}^4g_n(z)$. The expression $A(z)$ cannot be $0$ for any periodic point $z$ of period $5$ because then $g_i(z) = 0$ for some $1\leq{i}\leq{4}$, which implies $\phi^i(z)=\infty$, but $\infty$ is of period $2$ and cannot appear in a $5$-cycle. Now we take the resultant of of $\Phi_5^*(a_1,z)$ and $A(z)(t-tr(z))$ with respect to $z$, and find the following irreducible component:
\begin{equation}
\begin{split}
C_0: b_2^6t^6+8b_2^5t^5+(28b_2^4+4b_2^5)t^4+(54b_2^3+17b_2^4-8b_2^5)t^3+ \\
 (-27b_2^4+28b_2^3+60b_2^2)t^2+(36b_2+18b_2^2-18b_2^4-44b_2^3)t \\
+3b_2+9-22b_2^2+27b_2^4-b_2^3 = 0.
\end{split}
\end{equation}
The plane curve $C_0$ can be checked (e.g.\ using \magma) to be birational over $\RR$ to the conic $X^2+Y^2+Z^2=0$ in the projective plane $\PP^2_{[X:Y:Z]}$, which has no any $\QQ$-rational points. Therefore $C_0$ has only rational points that are singular, and these are easily checked to be $[0 : 1 : 0], [1 : 0 : 0]$, in homogeneous coordinates $[X:Y:Z]$ where $b_2 = X/Z$ and $z=Y/Z$. The rational points on $C_1$ consist of the points of indeterminacy of $\Psi$ together with rational preimages of the singular points of $C_0$. However, the preimages of the singular points are not defined over $\QQ$ and therefore $C_1$ has only $3$ rational points, which are the indeterminacy points of $\Psi$ : 
\[[0 : 1 : 1], [0 : 1 : 0], [1 : 0 : 0],\]
in homogeneous coordinates $[X:Y:Z]$ where $b_2 = X/Z$ and $z=Y/Z$. None of these points correspond to quadratic maps having the required properties.
\item[N=6] We start with normal form~\eqref{eq:cyc2-unique} in Lemma~\ref{lem:cyc2-unique}. Suppose $z$ is a point of period $6$. Therefore it is a root of the $6$-th dynatomic polynomial. Let $C_1$ be the curve defined by the equation $\Phi_6^*(b_2,z)=0$ in the affine plane $\AA^2_{[b_2,z]}$. Using \magma, one can check that this curve is absolutely irreducible and has genus $15$. 
Using the trace map as for $N=5$, we get the curve
\begin{equation*}
\begin{split}
C_0: b_2^9t^9 + 14b_2^8t^8 + 9b_2^8t^7 - 27b_2^8t^6 + 90b_2^7t^7 + 85b_2^7t^6 - \\
    211b_2^7t^5 - 222b_2^7t^4 + 315b_2^7t^3 + 346b_2^6t^6 + 358b_2^6t^5 - \\
    811b_2^6t^4 - 1123b_2^6t^3 + 1197b_2^6t^2 + 637b_2^6t + 343b_2^6 + \\
    872b_2^5t^5 + 854b_2^5t^4 - 1891b_2^5t^3 - 2288b_2^5t^2 + 2079b_2^5t + \\
    1519b_2^5 + 1488b_2^4t^4 + 1217b_2^4t^3 - 2705b_2^4t^2 - 2117b_2^4t + \\
    1953b_2^4 + 1713b_2^3t^3 + 1002b_2^3t^2 - 2182b_2^3t - 508b_2^3 + \\
    1278b_2^2t^2 + 411b_2^2t - 720b_2^2 + 558b_2t + 54b_2 + 108 = 0
\end{split}
\end{equation*}
which is obtained by reducing modulo the automorphism of order $6$ sending $(b_2,z)$ to $(b_2,\phi(z))$. This curve has genus $2$, and is birational to the curve
\begin{equation*}
C_h : y^2 = x^6 + 6x^5 - x^4 - x^2 - 2x - 3.
\end{equation*}
The curve $C_h$ has trivial torsion, and is of rank $1$. Using the implementation of the Chabauty method and Mordell--Weil sieve in \magma, we find all $5$ rational points on the curve. We lift these rational points back to $C_0$ to find the three points 
\[(-2/7, 7), (-9503/3136, 14420/9503), (-918476/6388711, 6323527/459238)\]
together with two points at infinity:
\[[0:1:0], [1:0:0],\]
where the coordinates are in $\PP^2_{[X:Y:Z]}$ with $b_0 = X/Z$ and $z=Y/Z$. 

We must also check for points of indeterminacy of the birational map $C_0\to{C_h}$, but in fact it turns out that there are no $\QQ$-rational indeterminacy points, so $C_0$ has exactly five $\QQ$-rational points. For each of the three finite points, we can recover the corresponding quadratic map from the first coordinate. For all three maps, the $6$-th dynatomic polynomial has a degree $6$ factor with cyclic Galois group of order $6$ but has no $\QQ$-rational roots.
\end{description}
\end{proof}

\begin{prop}\label{prop:critical5}
Let $\phi$ be a quadratic map defined over $\mathbb{Q}$, then $\phi$ has no $\mathbb{Q}$-rational periodic critical point of period $5$.
\end{prop}

\begin{proof}
Suppose that $P\in\mathbb{P}^1(\mathbb{Q})$ is periodic critical point of period $5$ of $\phi$. We can always conjugate $\phi$ so that the critical point $P$ is $0$, and that $\phi(0)=\infty$ and $\phi(\infty)=1$. We get that $\phi$ is of the form
\begin{equation*}
\phi(z)=\frac{a_2z^2+a_1z+a_0}{a_2z^2}=F(z)/G(z).
\end{equation*}
For $0$ to be periodic of period $5$ we must have $\phi^3(1)=0$. For this to happen we must have
\begin{equation*}
\begin{split}
a_2^5+2a_1^5+a_0^5+33a_2^3a_1a_0+53a_2^2a_1^2a_0+44a_2^2a_1a_0^2+\\
35a_2a_1^3a_0+42a_2a_1^2a_0^2+23a_2a_1a_0^3+8a_1^4a_0+13a_1^3a_0^2+\\
11a_1^2a_0^3+11a_2a_1^4+5a_2a_0^4+7a_2^4a_1+7a_2^4a_0+18a_2^3a_1^2+\\
15a_2^3a_0^2+21a_2^2a_1^3+12a_0^3a_2^2+5a_1a_0^4=0.
\end{split}
\end{equation*}
This equation defines a curve $C$ in $\mathbb{P}^2_{[a_0:a_1:a_2]}$. Using \magma, one can check that this curve has genus $1$. A quick search yields three $\mathbb{Q}$-rational points:
\[
[0:-1:1], [-1:1:0], [0:-\frac{1}{2}:1].
\]
The point $[0:-\frac{1}{2}:1]$ is nonsingular, and we can transform this curve to the elliptic curve
\begin{equation*}
E: y^2 + xy + y = x^3 - x^2 - x
\end{equation*}
sending $[0:-\frac{1}{2}:1]$ to $\mathcal{O} = [0 : 1 : 0]$, the unit element of $E$. The elliptic curve $E$ is curve $17a4$ in Cremona's database. It has Mordell--Weil rank of $0$, and a torsion subgroup isomorphic to $\mathbb{Z}/4\mathbb{Z}$. The four rational points on the curve are
\begin{equation*}
\mathcal{O} = [0 : 1 : 0],[ 0 : 0 : 1],  [0 : -1 : 1], [1 : -1 : 1].
\end{equation*}

We denote by $\varphi$ the birational map mapping $C$ to $E$, and by $\psi$ its inverse. These maps can be calculated using \magma and are too lengthy to reproduce here. We denote $S=\Ind(\phi)$ and $T=\Ind(\psi)$, the indeterminacy sets of $\phi$ and $\psi$. Then outside $\phi^{-1}(T)\cup{S}\subset{C}$ and $\psi^{-1}(S)\cup{T}\subset{E}$ we have $\phi$ and $\psi$ are isomorphisms. One can check (using \magma for example) that the only $\mathbb{Q}$-rational points in these two subsets are the seven points listed previously. Thus $C$ contains no other $\mathbb{Q}$-rational points, since otherwise $E$ would as well. We can thus conclude that there are no quadratic maps $\phi$ defined $\mathbb{Q}$ having a periodic critical point of period $5$, since the three points $[0:-1:1], [-1:1:0], [0:-\frac{1}{2}:1]$ correspond to degenerate maps.
\end{proof}

\section{Graph tables}

\begin{table}[h]
\begin{center}
\caption{Realizable PCF quadratic maps with a $\mathbb{Q}$-rational periodic critical point of period 2 (taken from  Lukas, Manes and Yap~\cite{article:lukas-manes-yap2014})}\label{table:r2-pcf}
\begin{tabular}{|c|c|c|}\hline
ID& $\phi(z)$						& Preperiodicity graph\\ 
\hline \hline 
P2P1 & $z^2-1$			& \xygraph{ 
		!{<0cm,0cm>;<2cm,0cm>:<0cm,1cm>::} 
		!{(0,0) }*+{\bullet_{\infty}}="a" 
		!{(1,0) }*+{\bullet_{1}}="b" 
		!{(2,0) }*+{\bullet_{0}}="c" 
		!{(3,0) }*+{\bullet_{-1}}="d" 
		"a":@(l,ru) "a"
		"b":"c"
		"d":@/_/"c"
		"c":@/_/"d"
	}\\
\hline

P2P2 & $\displaystyle\frac{-1}{4z^2-4z}$	&  \xygraph{ 
		!{<0cm,0cm>;<2cm,0cm>:<0cm,1cm>::} 
		!{(0,0) }*+{\bullet_{1/2}}="a" 
		!{(1,0) }*+{\bullet_{1}}="b" 
		!{(2,0) }*+{\bullet_{\infty}}="c" 
		!{(3,0) }*+{\bullet_{0}}="d" 
		"a":"b" 
		"b":"c" 
		"c":@/_/"d" 
		"d":@/_/"c" 
	}  \\
\hline

P2P3 & $\displaystyle\frac{-4}{9z^2-12z}$	&\xygraph{ 
		!{<0cm,0cm>;<1.5cm,0cm>:<0cm,1cm>::} 
		!{(0,-0.25) }*+{\bullet_{2/3}}="a" 
		!{(1,-0.25) }*+{\bullet_{1}}="b" 
		!{(2,-0.25) }*+{\bullet_{4/3}}="c" 
		!{(3,-0.25) }*+{\bullet_{\infty}}="d" 
		!{(4,-0.25) }*+{\bullet_{0}}="e" 
		!{(2,0.75) }*+{\bullet_{1/3}}="f"
		!{(2,1.1) }*+{\ }="x"
		"a":"b" 
		"b":"c" 
		"c":"d"
		"f":"c"
		"d":@/_/"e" 
		"e":@/_/"d" 
	} \\
\hline

P2P4 & $\displaystyle\frac{3 z^2-4 z+1}{1-4 z}$	 		& \xygraph{ 
		!{<0cm,0cm>;<2cm,0cm>:<0cm,1cm>::} 
		!{(0,0.5) }*+{\bullet_{1/2}}="a" 
		!{(1,0.5) }*+{\bullet_{1/4}}="b" 
		!{(2,0.5) }*+{\bullet_{\infty}}="c" 
		!{(1,-0.5) }*+{\bullet_{-2}}="d" 
		!{(2,-0.5) }*+{\bullet_{-1}}="e" 
		!{(0,-0.5) }*+{\bullet_{1/3}}="f" 
		"a":"b" 
		"b":"c" 
		"c":@(r,lu) "c"
		"d":@/_/"e" 
		"a":"b" 
		"e":@/_/"d" 
		"f":"d" 
	} \\ 
\hline

P2P5 & $\displaystyle 2/z^2$			&
\xygraph{
        !{<0cm,0cm>;<2cm,0cm>:<0cm,1cm>::}
               !{(1,0) }*+{\bullet_{0}}="c"
               !{(2,0)}*+{\bullet_{\infty}}="d"
       "c":@/_/"d"
       "d":@/_/"c"
    }\\
\hline

P2P6 & $\displaystyle 1/z^2$			&
 \xygraph{
        !{<0cm,0cm>;<2cm,0cm>:<0cm,1cm>::}
        !{(0,0) }*+{\bullet_{1}}="a"
        !{(-1,0) }*+{\bullet_{-1}}="b"
         !{(1,0) }*+{\bullet_{0}}="c"
         !{(2,0)}*+{\bullet_{\infty}}="d"
        "a" :@(r,lu)  "a"
        "b":"a"
       "c":@/_/"d"
       "d":@/_/"c"
    } \\ 
 \hline
\end{tabular}
\end{center}
\end{table}

\begin{table}[h]
\begin{center}
\caption{Realizable non-PCF quadratic maps with a $\mathbb{Q}$-rational periodic critical point of period 2}\label{table:r2}
\begin{tabular}{|c|c|c|}\hline

ID& $\phi(z)$						& Preperiodicity graph  \\ 
\hline \hline 

R2P1 & $\displaystyle\frac{z+1}{z^2}$			& \xygraph{ 
		!{<0cm,0cm>;<1cm,0cm>:<0cm,1cm>::} 
		!{(-2,0) }*+{\bullet_{-1}}="a" 
		!{(0,0) }*+{\bullet_{0}}="b" 
		!{(2,0) }*+{\bullet_{\infty}}="c" 
		"a":"b"
		"b":@/_/"c"
		"c":@/_/"b"
	} 
	\\ 
\hline

R2P2 & $\displaystyle\frac{-z+2}{z^2}$			& \xygraph{ 
		!{<0cm,0cm>;<1cm,0cm>:<0cm,1cm>::} 
		!{(-2,0) }*+{\bullet_{2}}="a" 
		!{(0,0) }*+{\bullet_{0}}="b" 
		!{(2,0) }*+{\bullet_{\infty}}="c" 
		!{(3,0) }*+{\bullet_{-2}}="d"
		!{(5,0) }*+{\bullet_{1}}="e" 		 
		"a":"b"
		"b":@/_/"c"
		"c":@/_/"b"
		"d":"e"
		"e":@(r,lu) "e"
	} 
	\\ 

\hline

R2P3 & $\displaystyle\frac{6z-4}{3z^2}$			& \xygraph{ 
		!{<0cm,0cm>;<1cm,0cm>:<0cm,1cm>::} 
		!{(-2,-1) }*+{\bullet_{1}}="a" 
		!{(-2,1) }*+{\bullet_{2}}="b"
		!{(0,0) }*+{\bullet_{2/3}}="c"
		!{(2,0) }*+{\bullet_{0}}="d"
		!{(4,0) }*+{\bullet_{\infty}}="e"
		"a":"c"
		"b":"c"
		"c":"d"
		"d":@/_/"e"
		"e":@/_/"d"
	} 
	\\ 

\hline

R2P4 & $\displaystyle\frac{3z-2}{z^2}$			& \xygraph{ 
		!{<0cm,0cm>;<1cm,0cm>:<0cm,1cm>::} 
		!{(-2,0) }*+{\bullet_{2/3}}="a" 
		!{(0,0) }*+{\bullet_{0}}="b" 
		!{(2,0) }*+{\bullet_{\infty}}="c" 
		!{(3,-1) }*+{\bullet_{2}}="d"
		!{(4,0) }*+{\bullet_{1}}="e" 		 
		!{(5,-1) }*+{\bullet_{1/2}}="f"
		!{(6,0) }*+{\bullet_{-2}}="g" 		 		
		"a":"b"
		"b":@/_/"c"
		"c":@/_/"b"
		"d":"e"
		"e":@(r,lu) "e"
		"f":"g"
		"g":@(r,lu) "g"		
	} 
	\\ 

\hline

R2P5 & $\displaystyle\frac{-z-1}{6z^2}$			& \xygraph{ 
		!{<0cm,0cm>;<1cm,0cm>:<0cm,1cm>::} 
		!{(-2,-1) }*+{\bullet_{-1/2}}="a" 
		!{(-2,1) }*+{\bullet_{1}}="b"
		!{(0,0) }*+{\bullet_{-1/3}}="c"
		!{(1,-1) }*+{\bullet_{1/2}}="d"		
		!{(2,0) }*+{\bullet_{-1}}="e"
		!{(4,0) }*+{\bullet_{0}}="f"
		!{(6,0) }*+{\bullet_{\infty}}="g"		
		"a":"c"
		"b":"c"
		"c":"e"
		"d":"e"
		"e":"f"
		"f":@/_/"g"
		"g":@/_/"f"
	} 
	\\ 

\hline

R2P6 & $\displaystyle\frac{5z-2}{3z^2}$			& \xygraph{ 
		!{<0cm,0cm>;<1cm,0cm>:<0cm,1cm>::} 
		!{(-2,0) }*+{\bullet_{2/5}}="a" 
		!{(0,0) }*+{\bullet_{0}}="b"
		!{(2,0) }*+{\bullet_{\infty}}="c"
		!{(3,-1) }*+{\bullet_{1/2}}="d"		
		!{(3,1) }*+{\bullet_{2}}="e"
		!{(5,0) }*+{\bullet_{2/3}}="f"
		!{(6,0) }*+{\bullet_{1}}="g"		
		"a":"b"
		"b":@/_/"c"
		"c":@/_/"b"
		"d":"f"
		"e":"f"
		"f":"g"
		"g":@(r,lu) "g"		
	} 
	\\ 

\hline

R2P7 & $\displaystyle\frac{7z-6}{z^2}$			& \xygraph{ 
		!{<0cm,0cm>;<1cm,0cm>:<0cm,1cm>::} 
		!{(-2,0) }*+{\bullet_{6/7}}="a" 
		!{(0,0) }*+{\bullet_{0}}="b" 
		!{(2,0) }*+{\bullet_{\infty}}="c" 
		!{(3,-1) }*+{\bullet_{6}}="d"
		!{(4,1) }*+{\bullet_{1}}="e" 		 
		!{(4,-1) }*+{\bullet_{2/3}}="f"
		!{(5,1) }*+{\bullet_{-3}}="g" 		 		
		!{(5,-1) }*+{\bullet_{3/2}}="h"
		!{(6,1) }*+{\bullet_{2}}="i" 		 				
		"a":"b"
		"b":@/_/"c"
		"c":@/_/"b"
		"d":"e"
		"e":@(r,lu) "e"
		"f":"g"
		"g":@(r,lu) "g"		
		"h":"i"
		"i":@(r,lu) "i"		
	} 
	\\ 

\hline

\end{tabular}
\end{center}
\end{table}

\begin{table}[h]
\begin{center}
\caption{Realizable quadratic maps with a $\mathbb{Q}$-rational periodic critical point of period 3} \label{table:r3}
\begin{tabular}{|c|c|c|}\hline

ID& $\phi(z)$						& Preperiodicity graph  
\\ 
\hline \hline 
R3P0 & $\displaystyle\frac{1}{(z-1)^2}$	&  \xygraph{ 
		!{<0cm,0cm>;<2cm,0cm>:<0cm,1cm>::} 
		!{(0.7,0.5) }*+{\bullet_{\infty}}="a" 
		!{(0.7,-0.5) }*+{\bullet_{0}}="b" 
		!{(0,0) }*+{\bullet_{1}}="c" 
		!{(-.75,0) }*+{\bullet_{2}}="d"
		"a":@/^/"b"
    		"b":@/^/"c"
      		"c":@/^/"a"
		"d":"c"
	} 
	\\ 
 \hline

R3P1 & $\displaystyle\frac{2z^2-z-1}{2z^2}$	 	&  \xygraph{ 
		!{<0cm,0cm>;<2cm,0cm>:<0cm,1cm>::} 		
		!{(0,0) }*+{\bullet_{0}}="a" 
		!{(1,0) }*+{\bullet_{\infty}}="b" 
		!{(0.5,-1) }*+{\bullet_{1}}="c" 
		!{(-0.5,-1) }*+{\bullet_{-1/2}}="d"
		!{(1.5,0	) }*+{\bullet_{-1}}="e"		
		"a":@/^/"b"
    		"b":@/^/"c"
      		"c":@/^/"a"
    		"d":@/^/"a"
      		"e":@/^/"c"       
	} 
	\\ 
 \hline

R3P2 & $\displaystyle\frac{z^2+5z-6}{z^2}$	 	&  \xygraph{ 
		!{<0cm,0cm>;<2cm,0cm>:<0cm,1cm>::} 
		!{(0,0) }*+{\bullet_{0}}="a" 
		!{(1,0) }*+{\bullet_{\infty}}="b" 
		!{(0.5,-1) }*+{\bullet_{1}}="c" 
		!{(-0.5,-1) }*+{\bullet_{-6}}="d"
		!{(1.5,0	) }*+{\bullet_{6/5}}="e"		
		!{(2,-1) }*+{\bullet_{3}}="f"
		!{(3,0	) }*+{\bullet_{2}}="g"				
		"a":@/^/"b"
    		"b":@/^/"c"
      		"c":@/^/"a"
    		"d":@/^/"a"
      		"e":@/^/"c"
    		"f":"g"   
		"g":@(r,lu) "g"		    		   		
	} 
	\\ 
 \hline

R3P3 & $\displaystyle\frac{5z^2-7z+2}{5z^2}$	 	&  \xygraph{ 
		!{<0cm,0cm>;<2cm,0cm>:<0cm,1cm>::} 
		!{(0,0) }*+{\bullet_{0}}="a" 
		!{(1,0) }*+{\bullet_{\infty}}="b" 
		!{(0.5,-1) }*+{\bullet_{1}}="c" 
		!{(-0.5,-1) }*+{\bullet_{2/5}}="d"
		!{(1.5,0	) }*+{\bullet_{2/7}}="e"		
		!{(-2,1) }*+{\bullet_{2}}="f"
		!{(-2,-1) }*+{\bullet_{1/3}}="g"				
		"a":@/^/"b"
    		"b":@/^/"c"
      		"c":@/^/"a"
    		"d":"a"
      		"e":@/^/"c"
    		"f":"d"   
		"g":"d"		    		   		
	} 
	\\ 
 \hline

R3P4 & $\displaystyle\frac{3z^2-5z+2}{3z^2}$	 	&  \xygraph{ 
		!{<0cm,0cm>;<2cm,0cm>:<0cm,1cm>::} 
		!{(0,0) }*+{\bullet_{0}}="a" 
		!{(1,0) }*+{\bullet_{\infty}}="b" 
		!{(0.5,-1) }*+{\bullet_{1}}="c" 
		!{(-0.5,-1) }*+{\bullet_{2/3}}="d"
		!{(1.5,0) }*+{\bullet_{2/5}}="e"		
		!{(-0.5,-2) }*+{\bullet_{-2}}="f"
		!{(0.5,-2) }*+{\bullet_{2}}="g"		
		!{(1.5,-2) }*+{\bullet_{1/3}}="h"
		!{(2.5,-2) }*+{\bullet_{1/2}}="i"		
		"a":@/^/"b"
    		"b":@/^/"c"
      		"c":@/^/"a"
    		"d":@/^/"a"
      		"e":@/^/"c"   
      		"f":"g"
          		"g":@/^/"h"
      		"h":@/^/"g"   		
		"i": "h"		    		   		
	} 
	\\ 
 \hline

R3P5 & $\displaystyle\frac{5z^2-11z+6}{5z^2}$	 	&  \xygraph{ 
		!{<0cm,0cm>;<2cm,0cm>:<0cm,1cm>::} 
		!{(0.5,0) }*+{\bullet_{0}}="a" 
		!{(1.5,0) }*+{\bullet_{\infty}}="b" 
		!{(1,-1) }*+{\bullet_{1}}="c" 
		!{(0,-1) }*+{\bullet_{6/5}}="d"
		!{(2,0) }*+{\bullet_{6/11}}="e"		
		!{(1,-3) }*+{\bullet_{2/3}}="f"
		!{(2,-3	) }*+{\bullet_{2/5}}="g"		
		!{(3,-3) }*+{\bullet_{3}}="h"
		!{(4,-3	) }*+{\bullet_{-3/2}}="i"
		!{(0,-2) }*+{\bullet_{6}}="j"
		!{(0,-4) }*+{\bullet_{3/5}}="k"		
		"a":@/^/"b"
    		"b":@/^/"c"
      		"c":@/^/"a"
    		"d":@/^/"a"
      		"e":@/^/"c"   
      		"f":"g"
          		"g":@/^/"h"
      		"h":@/^/"g"   		
		"i": "h"
      		"j":"f"
      		"k":"f"				
	} 
	\\ 
 \hline

\end{tabular}
\end{center}

\end{table}

\begin{table}[h]
\begin{center}
\caption{Realizable quadratic maps with a $\mathbb{Q}$-rational periodic critical point of period 4}\label{table:r4}
\begin{tabular}{|c|c|c|}\hline
ID & $\phi(z)$						& Preperiodicity graph 
\\ 
\hline \hline 
R4P1 & $\displaystyle\frac{12z^2-11z+2}{12z^2}$	 	&  \xygraph{ 
		!{<0cm,0cm>;<2cm,0cm>:<0cm,1cm>::} 
		!{(0,0) }*+{\bullet_{0}}="a" 
		!{(1,0) }*+{\bullet_{\infty}}="b" 
		!{(1,-1) }*+{\bullet_{1}}="c" 
		!{(0,-1) }*+{\bullet_{1/4}}="d"
		!{(-1,1.5) }*+{\bullet_{2/3}}="e"
		!{(2,-1.5	) }*+{\bullet_{2/11}}="f"
		!{(-1,-1.5) }*+{\bullet_{2/9}}="g"		
		"a":@/^/"b"
    		"b":@/^/"c"
      		"c":@/^/"d"
       		"d":@/^/"a"     		
    		"e":"a"
      		"f":"c"
      		"g":"d"
	} 
	\\ 
 \hline
 
R4P2 & $\displaystyle\frac{7z^2+29z-30}{7z^2}$	 	&  \xygraph{ 
		!{<0cm,0cm>;<2cm,0cm>:<0cm,1cm>::} 
		!{(0,0) }*+{\bullet_{0}}="a" 
		!{(0.75,0) }*+{\bullet_{\infty}}="b" 
		!{(0.75,-1) }*+{\bullet_{1}}="c" 
		!{(0,-1) }*+{\bullet_{6/7}}="d"
		!{(-1,1.5) }*+{\bullet_{-5}}="e"
		!{(2,-1.5	) }*+{\bullet_{30/29}}="f"
		!{(-1,-1.5) }*+{\bullet_{-30}}="g"		
		!{(2,1	) }*+{\bullet_{15/7}}="h"
		!{(3,1) }*+{\bullet_{2}}="i"		
		"a":@/^/"b"
    		"b":@/^/"c"
      		"c":@/^/"d"
       		"d":@/^/"a"     		
    		"e":"a"
      		"f":"c"
      		"g":"d"
      		    		"h":"i"   
		"i":@(r,lu) "i"		    		   		
	} 
	\\ 
 \hline

R4P3 & $\displaystyle\frac{3z^2+z-2}{3z^2}$	 	&  \xygraph{ 
		!{<0cm,0cm>;<2cm,0cm>:<0cm,1cm>::} 
		!{(0,0) }*+{\bullet_{0}}="a" 
		!{(0.5,1) }*+{\bullet_{\infty}}="b" 
		!{(1,0) }*+{\bullet_{1}}="c" 
		!{(0.5,-1) }*+{\bullet_{2/3}}="d"
		!{(-1,0) }*+{\bullet_{-1}}="e"
		!{(2,0) }*+{\bullet_{2}}="f"
		!{(1.5,-1) }*+{\bullet_{-2}}="g"		
		!{(-2,1) }*+{\bullet_{1/2}}="h"
		!{(-2,-1) }*+{\bullet_{-2/3}}="i"		
		"a":@/^/"b"
    		"b":@/^/"c"
      		"c":@/^/"d"
       		"d":@/^/"a"     		
    		"e":"a"
      		"f":"c"
      		"g":"d"
      		"h":"e"
      		"i":"e"
	} 
	\\ 
 \hline

R4P4 & $\displaystyle\frac{15z^2-11z+2}{15z^2}$	 	&  \xygraph{ 
		!{<0cm,0cm>;<2cm,0cm>:<0cm,1cm>::} 
		!{(0,0) }*+{\bullet_{0}}="a" 
		!{(0.5,1) }*+{\bullet_{\infty}}="b" 
		!{(1,0) }*+{\bullet_{1}}="c" 
		!{(0.5,-1) }*+{\bullet_{2/5}}="d"
		!{(-1,0) }*+{\bullet_{1/3}}="e"
		!{(2,0) }*+{\bullet_{2/11}}="f"
		!{(1.5,-1) }*+{\bullet_{2/9}}="g"		
		!{(-1,-3) }*+{\bullet_{2}}="j"
		!{(0,-3) }*+{\bullet_{2/3}}="k"		
		!{(1,-3) }*+{\bullet_{1/5}}="l"
		!{(2,-3	) }*+{\bullet_{1/4}}="m"
		"a":@/^/"b"
    		"b":@/^/"c"
      		"c":@/^/"d"
       		"d":@/^/"a"     		
    		"e":"a"
      		"f":"c"
      		"g":"d"
			"j":"k"
			"k":@/^/"l"
      		"l":@/^/"k"   		
			"m": "l"		    		   		
	} 
	\\ 
 \hline

R4P5 & $\displaystyle\frac{35z^2-31z+6}{35z^2}$	 	&  \xygraph{ 
		!{<0cm,0cm>;<2cm,0cm>:<0cm,1cm>::} 
		!{(0,0) }*+{\bullet_{0}}="a" 
		!{(0.5,1) }*+{\bullet_{\infty}}="b" 
		!{(1,0) }*+{\bullet_{1}}="c" 
		!{(0.5,-1) }*+{\bullet_{2/7}}="d"
		!{(-1,0) }*+{\bullet_{3/5}}="e"
		!{(2,0) }*+{\bullet_{6/31}}="f"
		!{(1.5,-1) }*+{\bullet_{6/25}}="g"		
		!{(-2,1) }*+{\bullet_{2}}="h"
		!{(-2,-1) }*+{\bullet_{3/14}}="i"	
		!{(-1,-3) }*+{\bullet_{6}}="j"
		!{(0,-3) }*+{\bullet_{6/7}}="k"		
		!{(1,-3) }*+{\bullet_{1/5}}="l"
		!{(2,-3	) }*+{\bullet_{1/4}}="m"
		"a":@/^/"b"
    		"b":@/^/"c"
      		"c":@/^/"d"
       		"d":@/^/"a"     		
    		"e":"a"
      		"f":"c"
      		"g":"d"
      		"h":"e"
      		"i":"e"		
			"j":"k"
			"k":@/^/"l"
      		"l":@/^/"k"   		
			"m": "l"		    		   		
	} 
	\\ 
 \hline
\end{tabular}
\end{center}
\end{table}

\begin{table}[h]
\begin{center}
\caption{Inadmissible quadratic maps with a $\mathbb{Q}$-rational periodic critical point of period 2}\label{table:n2e}
\begin{tabular}{|c|c|c|}\hline

ID						& Preperiodicity graph  & genus\\ 
\hline \hline 

N2E1			& \xygraph{ 
		!{<0cm,0cm>;<1cm,0cm>:<0cm,1cm>::} 	
		!{(0,0) }*+{\bullet_{}}="a_1"
		!{(2,0) }*+{\bullet_{}}="b_1"
		!{(4,0) }*+{\bullet_{}}="c1"
		!{(-1,1) }*+{\bullet_{}}="a_2"
		!{(-1,-1) }*+{\bullet_{}}="b_2"		
		!{(5,0) }*+{\bullet_{}}="a3"
		!{(6,0) }*+{\bullet_{}}="b3"		
		"a_1":"b_1"
		"b_1":@/_/"c1"
		"c1":@/_/"b_1"
		"a_2":"a_1"
		"b_2":"a_1"
		"a3":"b3"
		"b3":@(r,lu) "b3"
	} & 1\\ 
\hline


N2E2			& \xygraph{ 
		!{<0cm,0cm>;<1cm,0cm>:<0cm,1cm>::} 
		!{(-2,-1) }*+{\bullet_{}}="a" 
		!{(-2,1) }*+{\bullet_{}}="b"
		!{(0,0) }*+{\bullet_{}}="c"
		!{(1,-1) }*+{\bullet_{}}="d"		
		!{(2,0) }*+{\bullet_{}}="e"
		!{(4,0) }*+{\bullet_{0}}="f"
		!{(6,0) }*+{\bullet_{}}="g"		
		!{(-3,1) }*+{\bullet_{}}="h"
		!{(-3,0) }*+{\bullet_{}}="i"
		"a":"c"
		"b":"c"
		"c":"e"
		"d":"e"
		"e":"f"
		"f":@/_/"g"
		"g":@/_/"f"
		"h":"b"
		"i":"b"
	} & 1\\ 
\hline

N2E3			& \xygraph{ 
		!{<0cm,0cm>;<1cm,0cm>:<0cm,1cm>::} 	
		!{(-3,0) }*+{\bullet_{}}="a_1"
		!{(-2,0) }*+{\bullet_{}}="b_1"
		!{(-1,0) }*+{\bullet_{}}="c1"
		!{(3,0) }*+{\bullet_{}}="a_2"
		!{(2,0) }*+{\bullet_{}}="b_2"		
		!{(1,1) }*+{\bullet_{}}="c2"
		!{(1,0) }*+{\bullet_{}}="d2"		
		!{(0,1) }*+{\bullet_{}}="e2"
		!{(0,0) }*+{\bullet_{}}="f2"				
		"a_1":"b_1"
		"b_1":@/_/"c1"
		"c1":@/_/"b_1"
		"a_2":@(r,lu) "a_2"
		"b_2":"a_2"
		"c2":"b_2"
		"d2":"b_2"
		"e2":"c2"
		"f2":"c2"		
	} & 1\\ 
\hline

N2E4			& \xygraph{ 
		!{<0cm,0cm>;<1cm,0cm>:<0cm,1cm>::} 
		!{(-2,0) }*+{\bullet_{}}="a" 
		!{(0,0) }*+{\bullet_{}}="b" 
		!{(2,0) }*+{\bullet_{}}="c" 
		!{(3,-1) }*+{\bullet_{}}="d"
		!{(4,1) }*+{\bullet_{}}="e" 		 
		!{(5,-1) }*+{\bullet_{}}="f"
		!{(6,1) }*+{\bullet_{}}="g" 		 		
		!{(2,-2) }*+{\bullet_{}}="j"
		!{(4,-2) }*+{\bullet_{}}="k"						
		"a":"b"
		"b":@/_/"c"
		"c":@/_/"b"
		"d":"e"
		"e":@(r,lu) "e"
		"f":"g"
		"g":@(r,lu) "g"		
		"j":"d"
		"k":"d"	
	} & 1 \\ 
\hline


N2E5			& \xygraph{ 
		!{<0cm,0cm>;<1cm,0cm>:<0cm,1cm>::} 
		!{(-2,0) }*+{\bullet_{}}="a" 
		!{(-2,1) }*+{\bullet_{}}="b"
		!{(0,0) }*+{\bullet_{}}="c"
		!{(0,-1) }*+{\bullet_{}}="d"		
		!{(2,0) }*+{\bullet_{}}="e"
		!{(4,0) }*+{\bullet_{}}="f"
		!{(6,0) }*+{\bullet_{}}="g"		
		!{(-2,-1) }*+{\bullet_{}}="h"
		!{(-2,-2) }*+{\bullet_{}}="i"
		"a":"c"
		"b":"c"
		"c":"e"
		"d":"e"
		"e":"f"
		"f":@/_/"g"
		"g":@/_/"f"
		"h":"d"
		"i":"d"
	} & 0\\ 
\hline

\end{tabular}
\end{center}
\end{table}

\newpage
\FloatBarrier


\def\cprime{$'$}
\providecommand{\bysame}{\leavevmode\hbox to3em{\hrulefill}\thinspace}
\providecommand{\MR}{\relax\ifhmode\unskip\space\fi MR }
\providecommand{\MRhref}[2]{%
  \href{http://www.ams.org/mathscinet-getitem?mr=#1}{#2}
}
\providecommand{\href}[2]{#2}

\end{document}